\documentclass{amsart}

\usepackage[margin=1.2in]{geometry}

\usepackage{amssymb}
\usepackage{amsmath}
\allowdisplaybreaks

\usepackage{xcolor}
\usepackage[all]{xy}
\usepackage{hyperref}
\usepackage{amssymb}
\usepackage{amsthm}
\usepackage{tikz-cd}

\newcommand{\F}{\mathbb F}

\newcommand{\Q}{\mathbb Q}
\newcommand{\Z}{\mathbb Z}

\newcommand{\fp}{\mathfrak p}
\newcommand{\fP}{\mathfrak P}

\newcommand{\co}{\mathcal{O}}

\newcommand{\Gal}{\mathrm{Gal}}

\newcommand{\ord}{\mathrm{ord}}

\newcommand{\Cl}{\mathrm{Cl}}

\newcommand{\hilbert}[3]{\Bigl(\dfrac{#1,#2}{#3}\Bigr)}

\numberwithin{equation}{section}

\theoremstyle{plain}
\newtheorem{thm}{Theorem}

\newtheorem{cor}[thm]{Corollary}
\newtheorem{lem}[thm]{Lemma}

\theoremstyle{definition}

\newtheorem*{Milovic's conjecture}{Milovic's conjecture}
\newtheorem*{Equivalent Form of the Conjecture}{Equivalent Form of the Conjecture}
\newtheorem*{conj*}{Conjecture}

\begin{document}
\title{On the 2-adic logarithm of units of certain totally imaginary quartic fields}

\author{Jianing Li}
\address{CAS Wu Wen-Tsun Key Laboratory of Mathematics,   University of Science and Technology of China, Hefei, Anhui 230026, China}
\email{lijn@ustc.edu.cn}

\subjclass[2010]{11R29, 11R27, 11R29}

\keywords{$2$-adic logarithm, units, class groups, pure quartic fields}

\maketitle

\begin{abstract}
 In this paper,  we prove a result on the $2$-adic logarithm of the fundamental unit of the field $\Q(\sqrt[4]{-q}) $, where $q\equiv 3\bmod 4$ is a prime.  When $q\equiv 15\bmod 16$, this result confirms a speculation of Coates-Li and has consequences for certain Iwasawa modules arising in their work.
\end{abstract}

\section{Introduction}

Let $q$ be any prime $\equiv 3 \bmod 4$, and define
\[
	K = \Q(\sqrt{-q}), \, \, \,  F = K(\sqrt[4]{-q}).
\]
Then there is a unique prime $\fP$ of $F$ lying above $2$ which is ramified in the extension $F/\Q$ (see Lemma~\ref{lem: ram} below), and we write $\ord_\fP$ for the usual order valuation at $\fP$.  Moreover, $K$ has odd class number, and it is not difficult to show that $F$ also has odd class number (see Lemma~\ref{lem: oddness of clno} below).  The unit group of $F$ has rank 1, and we write $\eta$ for a fundamental unit of $F$. We have $\eta \equiv 1 \bmod \fP$ when $q > 3$, so that the usual logarithmic series $\log_\fP(\eta)$ will converge in the completion $F_\fP$ of $F$ at $\fP$ (see Lemma~\ref{lem: oddness of clno} below, where we also point out how to deal with the slightly exceptional case of $q =3$).
We shall use elementary arguments to prove the following result.

\begin{thm}\label{thm: main} Let $q$ be any prime $\equiv 3 \bmod 4$.  Let $\eta$ be a fundamental unit of $F$, and let $\fP$ be the unique ramified prime of $F$ above 2. Then (1) If $q \equiv  3 \bmod 8$, we have $\ord_\fP(\log_\fP(\eta)) = 0$; (2) If $q \equiv  7 \bmod 16$, we have $\ord_\fP(\log_\fP(\eta)) = 2$; and (3) If $q \equiv  15 \bmod 16$, we have 
	$\ord_\fP(\log_\fP(\eta)) \geq 4$. 
\end{thm}

We first remark that assertions (1) and (2) can be viewed as an exact $\fP$-adic form of the Brauer-Siegel theorem as $q$ varies. Secondly, our motivation for proving the above theorem came from a recent paper of J. Coates and Y. Li \cite{CL}, which uses 2-adic arguments from Iwasawa theory to prove various non-vanishing theorems for the values at $s=1$ of the complex $L$-series of certain elliptic curves with complex multiplication. In fact, the results in \cite{CL} are concerned with the field $F^* = \Q(\sqrt{-\sqrt{-q}})$, but we note that the fields $F$ and $F^*$ are isomorphic extensions of  $\Q$, and so Theorem \ref{thm: main} remains valid with $F^*$ replacing $F$. Assume first that $q \equiv 7 \bmod 8$, so that $2$ splits in $K$, and let $\fp$ be the unique prime of $K$ lying below $\fP$. By class field theory, there is a unique extension $K_\infty/K$ with Galois group $\Gal(K_\infty/K) \xrightarrow{\sim} \Z_2$,   which is unramified outside the prime 
$\fp$. Define $F_\infty^* = F^*K_\infty$, and let $\Gamma = \Gal(F_\infty^*/F)$. Let $M(F_\infty^*)$ (resp. $M(F^*)$)
denote the maximal abelian $2$-extension of $F_\infty^*$ (resp. $F^*$) which is unramified outside the primes of $F_\infty^*$ (resp. $F^*$) lying above $\fp$. Let $X(F_\infty^*) = \Gal(M(F_\infty^*)/F_\infty^*)$. Now $M(F_\infty^*)$ is clearly a Galois extension of $F^*$, and hence, as always in Iwasawa theory \cite{I1}, $\Gamma$ will act on $X(F_\infty^*)$ by lifting inner automorphisms. Writing $X(F_\infty^*)_\Gamma$ for the $\Gamma$-coinvariants of $X(F_\infty^*)$, we see immediately that $X(F_\infty^*)_\Gamma = \Gal(M(F^*)/F^*_\infty)$. Moreover
we have $X(F_\infty^*) = 0$ if and only if $X(F_\infty^*)_\Gamma = 0$.  By global class field theory, the Galois group $\Gal(M(F^*)/F_\infty^*)$ is a finite group, and a classical theorem of Coates and Wiles (see \cite[Theorem 8.2]{CL}) shows that 
\begin{equation}\label{eq: coates-wiles}
	[M(F^*):F_\infty^*] = 2^{(\ord_\fP(\log_\fP(\eta)) - 2)/2},
\end{equation}
where $\eta$ now denotes a fundamental unit of the field  $F^*$. Now when $q \equiv 7 \bmod 16$, Coates and Li show in \cite{CL} by a simple Iwasawa theoretic argument based on Nakayama's lemma that $X(F_\infty^*) = 0$, whence it follows from \eqref{eq: coates-wiles}
that $\ord_\fP(\log_\fP(\eta)) = 2$. Based on numerical computations carried out by Zhibin Liang, they also conjecture in \cite{CL} that $\ord_\fP(\log_\fP(\eta)) \geq 4$ when 
$q \equiv  15 \bmod 16$, but say that they cannot  prove this conjecture by the arguments of Iwasawa theory. Thus our theorem above confirms their conjecture, as well as giving a new and simple proof of their result when $q \equiv 7 \bmod 16$. In fact, when combined with the arguments from Iwasawa theory given in \cite{CL}, our result shows that $X(F_\infty^*)$
is a free finitely generated $\Z_2$-module of strictly positive rank when $q \equiv 15 \bmod 16$. Let $B$ be the abelian variety defined over $K$, which is the restriction of scalars
from the Hilbert class field of $K$ to $K$ of the elliptic curve $A$, with complex multiplication by the ring of integers of $K$, which was first defined by Gross (an equation for this elliptic curve is recalled in \cite{CL}, p. 1). Then in fact, when $q \equiv 15 \bmod 16$, our result shows that either $B(F_\infty^*)$ contains a point of infinite order, or the Tate-Shafarevich group of $B/F_\infty^*$ contains a copy of $\Q_2/\Z_2$. When $q \equiv 3 \bmod 8$, none of the above Iwasawa theoretic arguments remain literally valid, because $2$ now remains prime in $K$. Nevertheless, we cannot help speculating whether assertion (1) of Theorem \ref{thm: main} for $F^*$ could somehow be used to attack the non-vanishing Conjecture 1.8 of \cite{CL}. However, our theorem has the following consequence for primes $q\equiv 3\bmod 8$.
\begin{cor}\label{cor: iwa module}
	Suppose $q\equiv 3\bmod 8$. Let $F_\infty$ be the compositum of all $\Z_2$-extensions of $F$. Let  $M(F)$
	denote the maximal abelian $2$-extension of $F$ which is unramified outside $\fP$.  Then $M(F)=F_\infty$ and $\Gal(M(F)/F)\cong \Z^3_2$.
\end{cor}

We end this Introduction with two unrelated remarks. Firstly, the arguments used to prove Theorem \ref{thm: main} break down completely for primes $q \equiv 1 \bmod 4$, because then both 
$K$ and $F$ have even class numbers.  Secondly, the elementary arguments given in the next section hinge on the following simple observations. Firstly, we use repeatedly the identity
\[
	\eta^2 \pm 1 = \eta(\eta \pm \eta^{-1}).\]
Secondly, since the prime $\fP$ has ramification index 2, we have $\ord_\fP(\log_\fP(w)) = \ord_\fP(w-1)$ for any element of $w$ of $F$ with $\ord_\fP(w-1) >2$.

\section{Proofs}

In this section, we present our elementary proof for Theorem~\ref{thm: main}. Next we prove Corollary~\ref{cor: iwa module} by using a standard result of class field theory. Finally, we give another very simple proof for Theorem~\ref{thm: main}(3) by the Coates-Wiles formula \eqref{eq: coates-wiles}.

\begin{lem}\label{lem: ram}
There exists a unique ramified prime ideal $\fP$ of $F$ above $2$ which has ramification index $2$ in the extension $F/\Q$.
\end{lem}

\begin{proof}
A number field is ramified at a rational prime if and only if its Galois closure is ramified at that prime. It follows that $F/\Q$ is ramified at $2$ since its Galois closure $F(\sqrt{-1})$ is clearly ramified at $2$. If $q\equiv 3\bmod 8$, then $2$ is inert in $K$.  Hence $\fp=2\co_K$ must be ramified in $F/K$, with ramification index 2.
Assume next that $q\equiv 7\bmod 8$. Then $2$ splits in $K$, say $2\co_K=\fp\bar{\fp}$. The prime ideal $\fp$ induces an embedding from $K$ to $\Q_2$. We fix the choice of $\sqrt{-q}$ such that $\sqrt{-q}\equiv 3\bmod 8\Z_2$ when $q\equiv 7\bmod 16$ and that $\sqrt{-q}\equiv 7\bmod 8\Z_2$ when $q\equiv 15\bmod 16$. Then $\fp$ is ramified in $F$. Note that $\bar{\fp}$ is inert in $F$ when $q\equiv 7\bmod 16$ and that $\bar{\fp}$ splits in $F$ when $q\equiv 15\bmod 16$. This proves the lemma.
\end{proof}

\begin{lem}\label{lem: oddness of clno}
$(1)$Assume $q>3$.  Then the norm $N(\eta)$ of $\eta$  from $F$ to $K$ is $1$ and  $\eta$ is congruent to $1$ modulo $\fP$.
		
$(2)$ The class number  $h$ of $F$ is odd.
\end{lem}

\begin{proof}
	Note that $N(\eta)$ is a unit of $K$ and hence $N(\eta)=\pm 1$.  Since $q\equiv 3\bmod 4$,  the quadratic Hilbert symbol in the local field $\Q_q(\sqrt{-q})$
	\[ \hilbert{-1}{\sqrt{-q}}{\Q_q(\sqrt{-q})}= \hilbert{-1}{q}{\Q_q}=-1.  \] It follows that $-1\notin N(F^\times)$.  In particular, $N(\eta)=1$.  
	
	If $q\equiv 7\bmod 8$, then $\co_F/\fP\cong \F_2$ by the above lemma. Hence $\eta\equiv 1\bmod \fP$ clearly. Suppose next that $q\equiv 3\bmod8$.  Note that the polynomial $(x+1)^2-\sqrt{-q}$ is Eisenstein in $K_\fp[x]$ where $K_\fp=\Q_2(\sqrt{3})$ is the completion of $K$ at $\fp=2\co_K$.  It follows that the ring of integers of $F$ is $\co_K[\sqrt[4]{-q}]$. Write $\eta = a+b\sqrt[4]{-q}$ with $a,b\in \co_K$. By (1), the conjugate of $\eta$ is $\eta^{-1}$ and hence $\eta+\eta^{-1}=2a\equiv 0\bmod \fP$. Thus $\eta\equiv 1\bmod \fP$ by the structure of the finite field $\co_F/\fP=\F_4$. This proves (1).

	 For (2), we first note that $K$ has odd class number by genus theory. The ambiguous class number formula \cite[Chapter 13, Lemma 4.1]{Lang90} states that for a cyclic extension $F/K$ of number fields, the order of the $\Gal(F/K)$-invariant subgroup of the ideal class group $\Cl_F$ of $F$ is given by:
	 \[ |\Cl^{\Gal(F/K)}_F|=|\Cl_K|\frac{\prod_{v}e_v}{[F:K][\co^\times_K:\co^\times_K\cap N(F^\times)]}.   \]
	 Here $\Cl_K$ is the ideal class group of $K$, the product runs over all the places of $K$ and $e_v$ is the ramification index of $v$ in $F/K$. In our case, the ramified places are $\sqrt{-q}\co_K$ and $\fp$. Recall that $\fp$ is the prime of $K$ lying below $\fP$. By (1), we know that $-1\notin N(F^\times)$. 
	 Applying the above formula gives $2\nmid |\Cl^{\Gal(F/K)}_F|$. Hence $2\nmid h=|\Cl_F|$ by Nakayama's lemma.  
\end{proof}
We remark that for $q=3$,  multiplying $\eta$ by a third root of unity if needed,  we can also assume that $\eta\equiv 1\bmod \fP$.

\begin{lem}\label{lem: main}

$(1)$  If $q\equiv 3\bmod 8$, then $\ord_\fP(\eta+\eta^{-1})=\ord_\fP(\eta-\eta^{-1})= 2$;
	
$(2)$  If $q\equiv 7\bmod 16$, then $\ord_\fP(\eta+\eta^{-1})= 4$. 

$(3)$  If $q\equiv 15\bmod 16$, then $\ord_\fP(\eta+\eta^{-1})\geq 6$.
\end{lem}

\begin{proof}[Proof of Lemma~\ref{lem: main}]
The ideas of the proofs are the same  for all cases.  We first consider the case $q\equiv 3\bmod 8$ which is slightly easier to handle. 
If $q=3$, then $\eta=\frac{\sqrt{-3}+1}{2}-\sqrt[4]{-3}$, and it is readily verified that (1) holds. Assume now that $q>3$.
We have $\fp=2\co_K=\fP^2$. Then $\fP=\gamma\co_F$ for some $\gamma\in \co_F$ since the class number $h$ of $F$ is odd. It follows that $\frac{\gamma^2}{2}$ is a unit of $\co_F$. Thus $\frac{\gamma^2}{2}=\pm\eta^k$ for some integer $k$. We claim that $k$ is odd. Indeed, if $k$ is even, we would have that $(\gamma\eta^{-k/2})^2=\pm2$, whence $F=K(\sqrt{\pm 2})$, which is a contradiction. This proves the claim. By replacing $\gamma$ by $\gamma\eta^{-\frac{k-1}{2}}$, we may assume that $\frac{\gamma^2}{2}$ is the fundamental unit $\eta$.  In the proof of part (2) of Lemma~\ref{lem: oddness of clno}, we have shown that $\co_F=\co_K[\sqrt[4]{-q}]$. Thus we can write $\gamma= a+b\sqrt[4]{-q}$ with $a,b\in \co_K$, whence
	\[ \eta=\frac{a^2+b^2\sqrt{-q}}{2} +ab\sqrt[4]{-q} \quad \text{ and }\quad N(\gamma)=a^2-b^2\sqrt{-q}=\pm2.\] 
In fact, one can show that $N(\gamma)=-2$ by computing the Hilbert symbols of $-2$ and $\sqrt{-q}$, but we will not need this finer result.
 We need to calculate $a\bmod 2 \in \co_K/{2\co_K}\cong \F_4$. It is easy to see that $a\not\equiv 0\bmod 2\co_K$. We claim that $a\not\equiv 1\bmod 2\co_K$.  Note that $\sqrt{-q}\equiv 1\bmod 2\co_K$. It follows that $a^2\equiv b^2\bmod 2\co_K$.  
 Suppose $a\equiv 1\bmod 2\co_K$. Then $a^2\equiv b^2\equiv 1\bmod 4\co_K$. This contradicts to the equality $N(\gamma)=\pm2$ and this proves the claim. Since $a\not\equiv 1\bmod 2\co_K$, we have $a^2+1\not\equiv 0\bmod 2\co_K$ by the structure of the finite field $\F_4$. Since $N(\eta)=1$, the conjugate of $\eta$ is $\eta^{-1}$.  We then have $\ord_\fP(\eta+\eta^{-1})=\ord_\fP(a^2+b^2\sqrt{-q})=\ord_\fP(2(a^2+1))=2 $ and
$  \ord_\fP(\eta-\eta^{-1})=\ord_\fP(2ab\sqrt[4]{-q})=2.$
This completes the proof for $q\equiv 3\bmod 8$.

Now we assume $q\equiv 7\bmod 8$ in the rest of the proof. We have $\fP^h=\gamma \co_F$ for some $\gamma \in \co_F$. Put $\pi = N(\gamma)\in \co_K$. 
The equalities of ideals  $\fp^h\co_F = \fP^{2h} =  \pi\co_F = \gamma^2 \co_F$ gives a unit $\frac{\gamma^2}{\pi}$ of $F$. We have $\frac{\gamma^2}{\pi}= \pm \eta^k$ for some odd integer $k$, for the same reason as in the case $q\equiv 3\bmod8$. As $\eta \equiv 1\bmod \fP$, we have $\ord_\fP(\pm\eta^k\pm \eta^{-k})=\ord_\fP (\eta+\eta^{-1})$. We may assume that $\frac{\gamma^2}{\pi}$ is the fundamental unit $\eta$.  Write $\gamma = a +b \sqrt[4]{-q}$ with $a,b\in K$. Then 
\[\eta = \frac{a^2+\sqrt{-q}b^2}{\pi} + \frac{2ab\sqrt[4]{-q}}{\pi}
 \quad  \text{ and } \quad 
a^2-\sqrt{-q}b^2=\pi.\]
From now on, we work in $F_\fP$, which is a quadratic extension of $K_\fp=\Q_2$. Recall that as in the proof of Lemma~\ref{lem: ram}, the embedding induced by $\fp$ is chosen so that $\sqrt{-q}\equiv 3\bmod 8$ when $q\equiv 7\bmod 16$ and that $\sqrt{-q}\equiv 7\bmod 8$ when $q\equiv 15\bmod 16$. Note that the ring of integers of $F_\fP$ is $\Z_2[\sqrt[4]{-q}]$. Since $\gamma$ is integral in $F_\fP$, we have $a,b\in \Z_2$. 
Since $\ord_\fp(\pi)=h$, we 
can write $\pi=2^h u$ with $u\in \Z^\times_2$. Note that one must have $\ord_2(a)=\ord_2(b)$. Otherwise, the valuation of $\pi=N_{F_{\fP}/{K_\fp}}(a+b\sqrt[4]{-q})$ at $2$ is even which contradicts to the fact that $h$ is odd. Also note that if $c,d\in \Z^\times_2$, then $N_{F_\fP/{K_\fp}}(c+d\sqrt[4]{-q})\equiv 2\bmod 4\Z_2$. It follows that $\ord_2(a)=\ord_2(b)=(h-1)/2$.
Because $\pi=N_{F_\fP/{K_\fp}}(\gamma)$ is a norm,  we conclude the following values of the Hilbert symbols
\[ \hilbert{2^hu}{\sqrt{-q}}{K_\fp}=\hilbert{2u}{3}{\Q_2}=1  \text{ if } q\equiv 7\bmod 16 \] and 
\[ \hilbert{2^hu}{\sqrt{-q}}{K_\fp}=\hilbert{2u}{7}{\Q_2}=1   \text{ if } q\equiv 15\bmod 16.  \]
This implies that $ u\equiv 3\bmod 4 \text{ if } q\equiv 7\bmod 16$ and  that $u\equiv 1\bmod 4  \text{ if }  q\equiv 15\bmod 16.$
Thus \[ \frac{\eta+\eta^{-1}}{2}= \frac{a^2+\sqrt{-q}b^2}{\pi} =  \frac{2a^2-\pi}{\pi} =  (\frac{a}{2^{\frac{h-1}{2}}})^2 u^{-1}-1 \equiv u^{-1}-1    \equiv\begin{cases}  2 \bmod 4 & \text{ if } q\equiv 7\bmod 16, \\
0\bmod 4 & \text{ if } q\equiv 15\bmod 16.
\end{cases}\]
This finishes the proof of Lemma~\ref{lem: main} by the fact $\ord_\fP(2)=2$. 
\end{proof}

\begin{proof}[Proof of Theorem~\ref{thm: main}]	
As we mentioned in the end of the introduction, the basic fact that $\ord_\fP(\log_\fP(x))=\ord_\fP(x-1)$ if $\ord_\fP(x-1)>2$ will be used.  For a proof, see \cite[Lemma 5.5]{Was97}. 
Assume $q\equiv 3\bmod 8$. Then $\ord_\fP(\eta^2+1)=\ord_\fP(\eta^2+\eta\eta^{-1})=\ord_\fP(\eta+\eta^{-1})=2$ and $\ord_\fP(\eta^2-1)=\ord_\fP(\eta^2-\eta\eta^{-1})=\ord_\fP(\eta-\eta^{-1})=2$. Hence $\ord_\fP(\eta^4-1)=4$. This gives $\ord_\fP\log_\fP(\eta^4)=4$. Thus $\ord_\fP(\log_\fP(\eta))=\ord_\fP\log_\fP(\eta^4)-\ord_\fP(4)=0$. This proves (1).

Assume $q\equiv 7\bmod 16$. We have $\ord_\fP(\eta^2+1)=\ord_\fP(\eta^2+\eta\eta^{-1})=\ord_\fP(\eta+\eta^{-1})=4$. Then $\ord_\fP(\eta^2-1)=\ord_\fP(\eta^2+1-2)=\ord_\fP(2)=2$. This gives $\ord_\fP(\eta^4-1)=6$. Thus $\ord_\fP(\log_\fP(\eta^4)) =\ord_\fP(\eta^4-1)=6$.  Hence $\ord_\fP(\log_\fP(\eta))=6-\ord_\fP(4)=2$. This proves (2).

Assume $q\equiv 15\bmod 16$. Then $\ord_\fP(\eta^4-1)=\ord_\fP(\eta^2+1)+\ord_\fP(\eta^2-1)\geq 6+2=8$. Then $\ord_\fP(\log_\fP(\eta^4))=\ord_\fP(\eta^4-1)\geq 8$.  Thus  $\ord_\fP(\log_\fP(\eta))\geq 4$. This completes the proof of Theorem~\ref{thm: main}.\end{proof}

\medskip

Now, we prove Corollary~\ref{cor: iwa module}, and we begin by recalling a classical result from global class field theory. Let $L$ be any number field, and $p$ be a prime number. For a prime ideal $v$ of $L$, 
let $U_{1,v}$ denote the principal units in the completion $L_v$ of $L$, and put $U_1=\prod_{v\mid p}U_{1,v}$. Let $\phi$ be the canonical embedding $L \hookrightarrow \prod_{v\mid p}L_v$. Denote by $\mathcal{E}_1$ the group of global units of $L$ whose images lie in $U_1$, and let $\overline{\phi(\mathcal{E}_1)}$ denote the closure of $\phi(\mathcal{E}_1)$ in $U_1$ under the $p$-adic topology. Let $H$ be the $p$-Hilbert class field of $L$. Finally let $M(L)$ be  the maximal abelian $p$-extension of $L$, which is unramified outside the primes of $L$ lying above $p$. Then the Artin map induces an isomorphism 
\[ U_1/{\overline{\phi(\mathcal{E}_1)}} \cong \Gal(M(L)/H).\] 
This is a standard consequence of global  class field theory (see, for example,  \cite[Theorem 13.4]{Was97}). Note that $U_1$ is a finitely generated $\Z_p$-module of rank $[L:\Q]$. Moreover, the  $\Z_p$-module ${\overline{\phi(\mathcal{E}_1)}}$ has rank $\leq r_1+r_2-1$, and  Leopoldt's conjecture asserts that this rank is always equal to $r_1+r_2-1$;  here $r_1$ and $r_2$ are the number of real and complex places of $L$, respectively. 

\begin{proof}[Proof of Corollary~\ref{cor: iwa module}]
We apply the above isomorphism to the field $F$ with $q\equiv 3\bmod 8$ and the prime $2$.
In this case, $U_1=1+\fP\co_{F_\fP}$ has $\Z_2$-rank $[F:\Q]=4$, and $\overline{\phi(\mathcal{E}_1)}=\overline{\langle \eta, -1\rangle}$ clearly has $\Z_2$-rank $1$.  Moreover, the $2$-Hilbert class field of $F$ is $F$ itself since $F$ has odd class number by Lemma 4. Thus we obtain an isomorphism of $\Z_2$-modules
\[  (1+\fP\co_{F_\fP})/    \overline{\langle \eta, -1\rangle} \cong \Gal(M(F)/F).   \]
In order to prove $M(F)=F_\infty$, it suffices to show that there is no nontrivial torsion element in the group on the left. Consider the commutative diagram with exact rows
\[\begin{tikzcd}
0   \ar[r] & \{\pm 1\}   \ar[r] \ar[d]& \overline{\phi(\mathcal{E}_1)}  \ar[r, "\log_\fP"]\ar[d] & \Z_2\log_\fP(\eta) \ar[r] \ar[d] &  0\\
0  \ar[r]  &\mu(1+\fP\co_{F_\fP}) \ar[r] & 1+\fP\co_{F_\fP}   \ar[r, "\log_\fP"] & \log_\fP(1+\fP\co_{F_\fP}) \ar[r] & 0. 
\end{tikzcd} \]
Here $\mu(1+\fP\co_{F_\fP}) $ is the group of roots of unity in $1+\fP\co_{F_\fP}$ which equals $\{\pm 1\}$ as one can check that $\sqrt{-1}\notin F_\fP$. Thus the logarithm induces an isomorphism \[(1+\fP\co_{F_\fP})/    \overline{\langle \eta, -1\rangle} \cong \log_\fP(1+\fP\co_{F_\fP})/\Z_2\log_\fP(\eta).\]
Since $\ord_\fP(2)=2$, it is clear from the logarithmic series that
$\log_\fP(1+\fP\co_{F_\fP})\subset \co_{F_\fP}$.  We claim that the $\Z_2$-module $\log_\fP(1+\fP\co_{F_\fP})/\Z_2\log_\fP(\eta)$ is free. Suppose not. Then there exists an element $a$ in $\log_\fP(1+\fP\co_{F_\fP})\subset \co_{F_\fP}$ but not in $\Z_2\log_\fP(\eta)$ such that $2a\in \Z_2\log_\fP(\eta)$. Write $2a=r\log_\fP(\eta)$ with $r\in \Z_2$.  Note that $r$ must be in $\Z^\times_2$. This would give $\ord_\fP(\log_\fP(\eta))=\ord_\fP(2a)>0$ which contradicts to Theorem~\ref{thm: main}.
 Thus we have that $\Gal(M(F)/F)\cong\log_\fP(1+\fP\co_{F_\fP})/\Z_2\log_\fP(\eta) $ is a free $\Z_2$-module of rank $3$ and hence $M(F)=F_\infty$. This completes the proof. \end{proof}

We end this paper by noting a second and very simple proof of Theorem~\ref{thm: main}(3).
Suppose $q\equiv 7\bmod 8$, so that $2$ splits in $K$, and recall that  $\fp$ is the restriction of $\fP$ to $K$. As before, let $M(F)$ be the maximal abelian $2$-extension which is unramified outside $\fP$. By class field theory and the fact that $F$ has odd class number \cite[Theorem 11]{CW}, we have
\[ (1+\fP\co_{F_\fP})/\overline{\langle\eta,-1 \rangle} \cong \Gal(M(F)/F).   \]
Suppose now $q\equiv 15\bmod 16$. The embedding $K\hookrightarrow K_\fp=\Q_2$ induced by $\fp$ makes that $\sqrt{-q}\equiv -1\bmod 8$ whence $F_\fP=\Q_2(\sqrt{-1})$.  Clearly $\sqrt{-1}$ is in $1+\fP\co_{F_\fP}$ but not in $\overline{\langle\eta,-1 \rangle}$. Thus $\Gal(M(F)/F)$ has an element of order $2$. Now let $F_\infty = FK_\infty$, where $K_\infty$ is the unique $\Z_2$-extension of $K$ unramified outside $\fp$. Since $\Gal(F_\infty/F)$ is a free $\Z_2$-module of rank $1$, it follows that $\Gal(M(F)/F_\infty)$ must contain the element of order $2$, and so $\Gal(M(F)/F_\infty) \neq 0$.  By the formula \eqref{eq: coates-wiles} of Coates-Wiles, it follows that we must have $\ord_\fP(\log_\fP(\eta))\geq 4$, as required.

\section*{Acknowledgments}
The author thanks Professor John Coates and Yongxiong Li for their carefully reading of this paper and  helpful suggestions.
The author is supported by the Fundamental
Research Funds for the Central Universities (No.~WK0010000058), NSFC (Grant No.~11571328) and by the Anhui Initiative in Quantum Information Technologies
(Grant No.~AHY150200).

\end{document}